\documentclass[a4paper]{amsart}

\usepackage[english]{babel}
\usepackage[latin1]{inputenc}
\usepackage{enumerate} 
\usepackage{latexsym} 
\usepackage{amsthm} 
\usepackage{amssymb} 
\usepackage{amsmath} 
\usepackage{tikz}
\usetikzlibrary{arrows}

\providecommand{\scr}{\mathcal} 
		
\newtheorem{theorem}{Theorem}[section]
\newtheorem{lemma}[theorem]{Lemma}
\newtheorem{corollary}[theorem]{Corollary}
\newtheorem{proposition}[theorem]{Proposition}
\newtheorem{example}[theorem]{Example}

\theoremstyle{definition}
\newtheorem{definition}[theorem]{Definition}

\newtheorem*{assumption*}{Standing Assumption}
\theoremstyle{remark}
\newtheorem{remark}[theorem]{Remark}

\def\newspan{\operatorname{span}}
\def\Cspan{\overline{\newspan}_\C}

\newcommand{\C}{\mathbb{C}}

\newcommand{\N}{\mathbb{N}}

\newcommand{\Z}{\mathbb{Z}}

\newcommand{\DD}{\mathcal{D}}

\newcommand{\XX}{\mathcal{X}}

\newcommand{\ZZ}{\mathcal{Z}}

\newcommand{\OO}{\mathcal{O}}

\DeclareMathOperator{\sgn}{sgn}

\numberwithin{equation}{section}

\newcommand{\eup}{essentially unique partition of the unit}
\newcommand{\nice}{a unital, commutative subring of $\C$ closed under complex conjugation that has an \eup}

\title[The Cuntz splice does not preserve $*$-isomorphism]{The Cuntz splice does not preserve $*$-isomorphism of Leavitt path algebras over $\Z$}

\author{Rune Johansen}
\address[Rune Johansen]{Department of Mathematical Sciences, University of Copenhagen, Denmark}
\email{rune@math.ku.dk}

\author{Adam P W S{\o}rensen}
\address[Adam P W S{\o}rensen]{Department of Mathematics, University of Oslo, Norway}
\email{apws@math.uio.no}

\date{\today}

\begin{document}

\begin{abstract}
We show that the Leavitt path algebras $L_{2,\Z}$ and $L_{2-,\Z}$ are not isomorphic as $*$-algebras. 
There are two key ingredients in the proof.
One is a partial algebraic translation of Matsumoto and Matui's result on diagonal preserving isomorphisms of Cuntz--Krieger algebras. 
The other is a complete description of the projections in $L_{\Z}(E)$ for $E$ a finite graph. 
This description is based on a generalization, due to Chris Smith, of the description of the unitaries in $L_{2,\Z}$ given by Brownlowe and the second named author. 
The techniques generalize to a slightly larger class of rings than just $\Z$.
\end{abstract}

\maketitle

\section{Introduction}

How are two Leavitt path algebras related if the underlying graphs are related by the so called Cuntz splice?
This is arguably one of the most important open problems in the theory of graph algebras.
More concretely, the question is whether the Cuntz splice preserves either the Morita equivalence class or the isomorphism class of the Leavitt path algebras considered as either rings, algebras, or $*$-algebras.
In this paper, we provide the first answer to any of these questions:
We show that there cannot exist a $*$-isomorphism between the Leavitt path algebras over the ring of integers induced by the graphs $E_2$ and $E_{2-}$ shown below. 

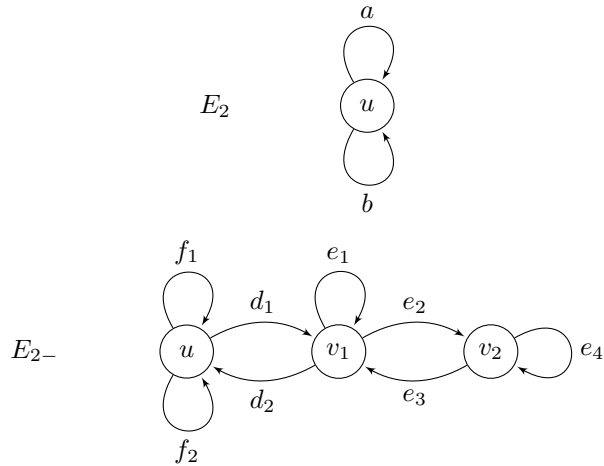
\begin{center}
\begin{tikzpicture}[shorten >=1pt]

\tikzset{vertex/.style = {shape=circle,draw,minimum size=2em}}
\tikzset{edge/.style = {->,> = latex'}}

\node at (0,0) {$E_2$};

\node[vertex] (u) at (2,0) {$u$};

\draw[edge] (u) to[loop, out=120, in=60, looseness=8] node [above] {$a$} (u);
\draw[edge] (u) to[loop, out=240, in=300, looseness=8] node [below] {$b$} (u);

\end{tikzpicture}
\end{center}

\begin{center}
\begin{tikzpicture}[shorten >=1pt]

\tikzset{vertex/.style = {shape=circle,draw,minimum size=2em}}
\tikzset{edge/.style = {->,> = latex'}}

\node at (0,0) {$E_{2-}$};

\node[vertex] (u) at  (2,0) {$u$};
\node[vertex] (v1) at  (4,0) {$v_1$};
\node[vertex] (v2) at  (6,0) {$v_2$};

\draw[edge] (u) to[loop, out=120, in=60, looseness=8] node [above] {$f_1$} (u);
\draw[edge] (u) to[loop, out=240, in=300, looseness=8] node [below] {$f_2$} (u);
\draw[edge] (u) to[bend left] node [midway, above] {$d_1$} (v1);

\draw[edge] (v1) to[loop, out=120, in=60, looseness=8] node[above] {$e_1$} (v1);
\draw[edge] (v1) to[bend left] node[midway, above] {$e_2$} (v2);
\draw[edge] (v1) to[bend left] node[midway, below] {$d_2$} (u);

\draw[edge] (v2) to[loop, out=30, in=-30, looseness=8] node[right] {$e_4$} (v2);
\draw[edge] (v2) to[bend left] node[below] {$e_3$} (v1);

\end{tikzpicture}
\end{center}

\noindent Some authors prefer the name $R_2$ (``rose with two petals'') for the graph we call $E_2$.

The interest in isomorphism and Morita equivalence of Leavitt path algebras of graphs related by the Cuntz splice stems from a desire to understand whether Leavitt path algebras can be classified through algebraic analogues of the celebrated results of R\o rdam (\cite{RordamClassificationOfCK}) and Kirchberg--Phillips (\cite{KirchbergClassification, KirchbergPhillipsEmbedding, PhillipsClassification}).
We now give a short overview of some of the historic developments in the field.
      
To each irreducible non-permutation matrix $A \in M_n(\N)$, one can associate a shift of finite type, $\XX_A$. 
For such shift spaces, the Bowen--Franks invariant consisting of the Bowen--Franks group $BF(A) = \Z^n / (I-A)\Z^n$ and the sign of the determinant of $I - A$ is a complete invariant of flow equivalence \cite{BowenFranks, FranksFlowEquivalence, ParrySullivan}.
To each non-permutation matrix $A \in M_n(\N)$, one can also associate  a $C^*$-algebra $\OO_A$, called the Cuntz--Krieger algebra, and the stable isomorphism class of $\OO_A$ is a flow invariant for $\XX_A$.
In fact, $K_0(\OO_A) \cong BF(A)$.
Considering the flow classification of $\XX_A$, one might wonder if $\OO_A$ also contains enough information to determine the other half of the Bowen--Franks invariant, $\sgn(\det(I-A))$. On its own, $\OO_A$ does not: R{\o}rdam proved that simple Cuntz--Krieger algebras are classified up to stable isomorphism by their $K_0$-groups \cite{RordamClassificationOfCK}.

Leavitt path algebras were introduced independently in \cite{AbramsArandaPino05} and \cite{AraMorenoPardo} as universal algebras over some field $K$ with relations given by the underlying graph.
They are purely algebraic analogues of graph $C^*$-algebras, which, in turn, are generalizations of Cuntz--Krieger algebras, see for instance \cite{BatesPaskRaeburnSzymanski, EnomotoWatatani, FowlerLacaRaeburn, KumjianPaskRaeburn, KumjianPaskRaebrunRenault, RaeburnBook}.
More generally, work of Tomforde shows how to define a Leavitt path algebra over any unital, commutative ring \cite{TomfordeLeavittOverRing}.
A good overview of the theory of Leavitt path algebras can be found in the survey paper \cite{AbramsFirstDecade}.

In \cite{AbramsLoulyPardoSmith}, it is shown that if $\XX_A$ is flow equivalent to $\XX_B$ for finite irreducible non-permutation matrices $A$ and $B$, then $L_K(E_A)$ is Morita equivalent to $L_K(E_B)$ for any field $K$. 
It is also shown that the $K_0$-group of $L_K(A)$ coincides with the Bowen--Franks group of $A$, so that a Leavitt path algebra is classified by its $K_0$-group and the sign of the determinant. 
However, it is an open problem whether the sign of the determinant is an isomorphism invariant, and this leaves the theory of Leavitt path algebras in a situation very similar to the situation in the theory of Cuntz--Krieger algebras prior to R\o rdam's classification result.

For the graphs $E_2$ and $E_{2,-}$ defined above, it is straightforward to check that the associated Bowen--Franks groups are trivial.
However, the signs of the determinants are different, and hence, the associated shifts of finite type are not flow equivalent. 
The graph $E_{2-}$ is constructed from $E_2$ by attaching two extra vertices and the associated edges. 
This gluing operation is called ``performing a Cuntz splice'', and it has played an important role in the study Cuntz--Krieger algebras, graph $C^*$-algebras, and Leavitt path algebras \cite{AbramsLoulyPardoSmith, ERRS, RestorffClassificationOfCK, RordamClassificationOfCK, RuizTomforde,  SorensenGeometricClassification}. 
It serves to flip the sign of the determinant while keeping the Bowen--Franks group fixed.
The graph $E_2$ is the simplest possible example of a graph where this operation can be performed without obviously changing Morita equivalence class.

In R{\o}rdam's work, it was shown that $C^*(E_2) \cong C^*(E_{2-})$, and this was then used to derive the general non-reliance on the sign of the determinant \cite{RordamClassificationOfCK}. 
If we wish to pursue a similar approach to the classification of Leavitt path algebras and understand whether there can exists an algebraic Kirchberg--Phillips theorem in this setting, then we must know whether the Leavitt path algebras associated to $E_2$ and $E_{2-}$ are isomorphic.

Our main results are Theorem \ref{thm: algebraic matsumoto matui}, Theorem \ref{thm: isomorphism same sign}, and Corollary \ref{cor: only diagonal projections}.
All these results concern a subalgebra of $L_{R}(E)$ called the diagonal, see Definition \ref{def: diagonal}.
Specifically, Theorem \ref{thm: algebraic matsumoto matui} states that for a subring $R$ of $\C$ that is closed under complex conjugation and graphs $E$ and $F$ satisfying standard assumptions, there can only exist a diagonal preserving $*$-isomorphism between $L_R(E)$ and $L_R(F)$ if the Bowen--Franks determinants are equal.
This is a partial algebraic analogue of a result of Matsumoto and Matui (\cite{MatsumotoMatui}) and it is derived directly from their result. 
Theorem \ref{thm: isomorphism same sign} shows that if all projections in $L_{R}(F)$ are elements of the diagonal, then $L_{R}(E)$ and $L_{R}(F)$ can only be isomorphic as $*$-algebras if their Bowen--Franks determinants are equal.
Corollary \ref{cor: only diagonal projections} shows that all projections of $L_{\Z}(E)$ are elements of the diagonal if $E$ is a finite graph.

In the initial version of this paper, we only proved that all projections in $L_{\Z}(E_2)$ are diagonal, which was sufficient to let us conclude that $L_{\Z}(E_2)$ and $L_{\Z}(E_{2-})$ are not $*$-isomorphic.
Our arguments were based on a complete description of the unitaries in $L_{\Z}(E_2)$ given in \cite{BS}.
Since then, Chris Smith has shown us how to extend the description of the unitaries in $L_{\Z}(E_2)$ from \cite{BS} to cover all finite graphs, see Proposition \ref{prop: standard form unitaries}, and we have realized how to extend this from $\Z$ to slightly more general rings.
Therefore, we can show that all projections in $L_{R}(E)$ are diagonal whenever $E$ is  a finite graph and $R$ is \nice\ (see Definition \ref{def:PropertyX}). 
In particular, this holds if $R = \Z$.

The three main results combine to show that if $E,F$ are finite, strongly connected, essential and non-trival graphs (see Definition \ref{def: graph concepts}), if $R$ is \nice\ and if $L_{R}(E)$ and $L_{R}(F)$ are $*$-isomorphic, then $\sgn(I - A_E) = \sgn(I - A_F)$ (Corollary \ref{cor: iso implies same sign}). 
Possible applications of the results to questions concerning Morita equivalence and Leavitt path algebras over fields are discussed in Remarks \ref{remark:fields} and \ref{remark:morita}.

Since the initial version of this paper appeared on the arXiv, there have been some interesting developments related to the problem discussed here. 
In \cite[Corollary 4.8]{AraBosaHazratSims} the authors use groupoid techniques to show that when $R$ is an integral domain, there can be no diagonal preserving \emph{ring} isomorphism from $L_{2,R}$ to $L_{2-,R}$.
In \cite[Corollary 4.7]{CarlsenRuizSims}, it is shown that $L_{2,\Z}$ to $L_{2-,\Z}$ are not stably isomorphic as $*$-algebras, in a suitable sense.

\section{Preliminaries}

\begin{definition}
A graph $E = (E^0, E^1, r, s)$ is a four tuple where $E^0$ is the vertex set, $E^1$ the edge set, and $r,s \colon E^1 \to E^0$ are the range and source maps.
\end{definition}

\noindent Given a graph $E$, we denote by $A_E$ its adjacency matrix, i.e.\ the $E^0 \times E^0$ matrix with $A(v,w) = |s^{-1}(v) \cap r^{-1}(w)|$. 
The following graph concepts will be needed below. 

\begin{definition}[{\cite{KumjianPaskRaeburn}}]
A graph $E$ satisfies Condition (L) if every cycle in $E$ has an exit. 
\end{definition}

\begin{definition} \label{def: graph concepts}
A graph $E$ is said to be:
\begin{itemize}
	\item \emph{strongly connected} if there is a path between any two vertices, 
	\item \emph{essential} if it has no sinks or sources, and
	\item \emph{trivial} if it is a single cycle with no other edges or vertices.
\end{itemize}
\end{definition}

\noindent Note that if $E$ is an essential and strongly connected graph then $A_E$ will be an irreducible matrix with no zero rows or columns.
Furthermore, if $E$ is non-trivial then $A_E$ will not be a permutation matrix. 

For a graph $E$, $E^n$ will denote the set of \emph{paths} of length $n$, and $E^* = \cup_{n \in \N} E^n$ will denote the collection of all finite paths. The range and source maps are extended to these sets in the natural way. Similarly, $E^\infty = \{ e_1 e_2 \cdots \mid e_i \in E^1 \textnormal{ and } r(e_i) = s(e_{i+1}) \}$ will denote the set of infinite paths. 
In the case of $E_2$, $E_2^\ast = \{a,b\}^* $ and $E_2^\infty = \{a,b\}^\N$. The \emph{cylinder set} of $\alpha \in E^\ast$ is defined to be
\begin{displaymath}  
\ZZ(\alpha) = \{ \alpha \xi \mid \xi \in E^\infty \textnormal{ and } \alpha \xi \in E^\infty\}.
\end{displaymath}
For $\alpha, \beta \in E^*$, $\ZZ(\alpha)$ and $\ZZ(\beta)$ are disjoint if and only if neither $\alpha$ nor $\beta$ extends the other. 

To each graph $E$, we associate a graph $C^*$-algebra and a collection of Leavitt path algebras as defined below.

\begin{definition}[{\cite[Definition 1]{FowlerLacaRaeburn}}]
Let $E$ be a graph. 
The graph $C^*$-algebra of $E$, $C^*(E)$, is the universal $C^*$-algebra generated by mutually orthogonal projections $\{ p_v \mid v \in E^0 \}$ and partial isometries $\{ s_e \mid e \in E^1 \}$ subject to the relations 
\begin{enumerate}[(i)]
 \item $s_e^* s_f = 0$, if $e \neq f$,
 \item $s_e^* s_e = p_{r(e)}$, 
 \item $s_e s_e^* \leq p_{s(e)}$, and,
 \item $p_v = \sum_{e \in s^{-1}(v)} s_e s_e^*$, if $s^{-1}(v)$ is finite and nonempty. 
\end{enumerate}

\end{definition}

\noindent Readers unfamiliar with the subject should be warned that there are two competing conventions for the definition of $C^*(E)$. This dichotomy stems from a necessary asymmetry between sinks and sources in the defining relations. 
Raeburn's monograph \cite{RaeburnBook} uses the other possible convention.

\begin{definition}[{\cite[Definitions 2.4 and 3.4]{TomfordeLeavittOverRing}}] \label{def:lpa}
Let $R$ be a commutative ring with unit and let $E$ be a graph. 
The Leavitt path algebra of $E$ over $R$ is the universal $R$-algebra generated by pairwise orthogonal idempotents $\{ v \mid v \in E^0\}$ and elements $\{e, e^* \mid e \in E^1\}$ satisfying 
\begin{enumerate}[(i)]
 \item $e^* f = 0$, if $e \neq f$,
 \item $e^* e = r(e)$, 
 \item $s(e) e = e = e r(e)$, 
 \item $e^* s(e) = e^* = r(e) e^*$, and, 
 \item $v = \sum_{e \in s^{-1}(v)} e e^*$, if $s^{-1}(v)$ is finite and nonempty. \label{item:summation_relation}
\end{enumerate}
\end{definition}

By \cite[Proposition 3.4]{TomfordeLeavittOverRing}, $L_R(E) = \newspan_{R}\{ \alpha \beta^* \mid \alpha, \beta \in E^*, r(\alpha) = r(\beta)\}$.
For any unital, commutative ring $R$, we can extend the map $\alpha \beta^* \mapsto \beta \alpha^*$ to a \emph{linear} involution on $L_R(E)$. 
If $R$ is a subring of $\C$ that is closed under complex conjugation, we will
extend $\alpha \beta^* \mapsto \beta \alpha^*$ to a \emph{conjugate linear} involution on $L_{R}(E)$ instead.
Throughout this paper, we are mainly interested in the latter case.

Since we consider $L_{R}(E)$ as a $*$-algebra, we prefer to work with self-adjoint idempotents rather than just idempotents.
Following the name conventions of $C^*$-algebras, we call these elements projections.

\begin{definition}
An element $p \in L_{R}(E)$ is called a \emph{projection} if $p = p^2 = p^*$.
\end{definition}

Each graph algebra contains a distinguished subalgebra called the diagonal:     

\begin{definition} \label{def: diagonal}
Let $E$ be a graph and $R$ a commutative ring with unit. 
We define 
\[
  \DD(C^*(E)) = \Cspan\{ s_\alpha s_\alpha^* \mid \alpha \in E^* \},
\]
and 
\[
 \DD(L_R(E)) = \newspan_{R}\{ \alpha \alpha^* \mid \alpha \in E^* \}.
\]
We refer to these sets as the diagonal of $C^*(E)$ and $L_R(E)$, respectively. 
\end{definition}

\noindent It is shown in \cite[Theorem 5.2]{HopenwasserPetersPower} and \cite[Theorem 3.7]{NagyReznikoff}) that $\DD(C^*(E))$ is a MASA (maximal abelian sub-algebra) in $C^*(E)$ if $E$ satisfies Condition (L).
We note that $\DD(L_R(E))$ is generated by projections.
It has recently been shown that if $E$ satisfies condition (L) then $\DD(L_{R}(E))$ is maximal abelian in $L_R(E)$ (see \cite[Proposition 3.12 and Theorem 3.9]{GilNasr} and \cite[Lemma 3.13]{AraBosaHazratSims}).

We are interested in subrings of $\C$ that behave like $\Z$ in the following very specific way.

\begin{definition}[Essentially unique partition of the unit] \label{def:PropertyX}
A unital (same unit) subring $R \subset \C$ closed under complex conjugation has an \emph{\eup} if whenever $\lambda_1, \lambda_2, \ldots, \lambda_n \in R$ satisfy
\[
  \sum_{i=1}^n \lambda_i \overline{\lambda_i} = \sum_{i=1}^n |\lambda_i|^2 = 1,  
\]
then all but one of the $\lambda_i$ is zero. 
\end{definition}

\begin{example} 
The following is an incomplete list of subrings of $\C$ with an \eup.
\begin{itemize}
 \item $\Z$. 
 \item $\Z + \sqrt{p}\Z$, for $p$ a prime. 
 \item The Gaussian integers, $\Z + i\Z$.
 \item The ring $\Z[\pi]$, $\Z + \pi\Z + (\pi^2)\Z + \cdots$.
\end{itemize}
\end{example}

\begin{remark} \label{rmk:PropX_nohalf}
Suppose $R \subset \C$ has an \eup. 
For each $n \neq 1$, $\frac{1}{n}\notin R$ since 
\[
  \sum_{k=1}^{n^2} \left( \frac{1}{n} \right)^2 = 1.
\]
Note that this means that no subfield of $\C$ has an \eup.
\end{remark}

\section{Matsumoto and Matui's Result for Graph Algebras}

In \cite{MatsumotoMatui}, Matsumoto and Matui show that under standard assumptions on matrices $A$ and $B$, there can only exist a diagonal preserving isomorphism between the Cuntz--Krieger algebras (see Definition \ref{def: Cuntz Krieger Algebra}) $\OO_A$ and $\OO_B$ if $\sgn(\det (I - A)) = \sgn( \det(I - B))$.
In this section, we will translate this part of Matsumoto and Matui's spectacular result into the world of graph algebras. In Section \ref{subsection:cstar}, the result is translated into a statement about graph $C^*$-algebras and in Section \ref{subsection:leavitt}, this is used to prove an algebraic analogue for Leavitt path algebras over subrings of $\C$. Both results are straightforward adaptations of those presented in \cite{MatsumotoMatui}.

\subsection{Graph $C^*$-algebras}\label{subsection:cstar}

First, recall the definition of a Cuntz--Krieger algebra and its diagonal (originally from \cite{CuntzKrieger}). 

\begin{definition} \label{def: Cuntz Krieger Algebra}
Let $A$ be an $N \times N$ matrix with entries in $\{0,1\}$. 
The Cuntz--Krieger algebra $\OO_A$ is the universal $C^*$-algebra generated by partial isometries $S_1, S_2, \ldots, S_N$ satisfying the relations
\begin{itemize}
 \item $\sum_{i=1}^N S_i S_i^* = 1$, and, 
 \item $S_i^* S_i = \sum_{j=1}^N A(i,j) S_j S_j^*$,
\end{itemize}
where $A(i,j)$ denotes the $ij$'th entry of $A$. 
The diagonal of $\OO_A$ is 
\[
  \DD(\OO_A) = \Cspan\{ S_{i_1} S_{i_2} \ldots S_{i_k} S_{i_k}^* \ldots S_{i_1}^* \mid k \in \N, i_1, i_2, \ldots i_k \in \{1,2,\ldots, N\}  \}. 
\]
\end{definition}

We observe that for two indicies $i,j$ we have 
\[
  S_i S_j = S_i S_i^* S_i S_j = S_i \left( \sum_{k=1}^N A(i,k) S_{k} S_{k}^* \right) S_j = A(i,j) S_i S_j.
\]
Hence, $S_i S_j = 0$ if $A(i,j) = 0$. 

Condition (L) for graphs was designed to be an equivalent of Cuntz and Krieger's Condition (I) (\cite{CuntzKrieger}).
We recall that a $\{0,1\}$ matrix will satisfy Condition (I) if it is not a permutation matrix. 

Graph $C^*$-algebras were originally defined to be generalizations of Cuntz--Krieger algebras, and there are strong connections between the two concepts. 
We have been unable to find a reference for the following lemma, but it is certainly well-known to researchers in the field.

\begin{lemma} \label{lem:CKisGraph}
Let $A$ be an irreducible $N \times N$ matrix with entries in $\{0,1\}$ that satisfies Condition (I). Let $E_A$ be the graph with $E_{A}^0 = \{1,2, \ldots, N\}$, $E_{A}^1 = \{ [ij] \mid A(i,j) = 1 \}$ and range and source maps given by 
\[
  r([ij]) = j \quad \quad \text{ and } \quad \quad s([ij]) = i. 
\]
There exists an isomorphism $\phi \colon C^*(E_A) \to \OO_{A}$ such that $\phi(\DD(C^*(E_A))) = \DD(\OO_A)$. 
\end{lemma}
\begin{proof}
Define $\phi$ on generators by 
\begin{align*}
 \phi(p_{i}) &= S_i S_i^*, \\
 \phi(s_{[ij]}) &= S_i S_j S_j^*. 
\end{align*}
By (the proof of) \cite[Proposition 4.1]{MannRaeburnSutherland}, $\phi$ is an isomorphism. 

Given a path $\alpha = [i_1 i_2] [i_2 i_3] \ldots [i_{k-1} i_k]$ in $E_A$ we see that 
\begin{align*}
  \phi(s_\alpha) &= (S_{i_1} S_{i_2} S_{i_2}^*) (S_{i_2} S_{i_3} S_{i_3}^*) \cdots (S_{i_{k-1}} S_{i_k} S_{i_k}^*) \\
	       &= S_{i_1} (S_{i_2} S_{i_2}^* S_{i_2}) (S_{i_3} S_{i_3}^* S_{i_3}) \cdots (S_{i_{k-1}} S_{i_{k-1}}^* S_{i_{k-1}}) S_{i_k} S_{i_k}^* \\
	       &= S_{i_1} S_{i_2} \cdots S_{i_k} S_{i_k}^*.
\end{align*}
Hence
\[
  \phi(s_\alpha s_\alpha^*) = S_{i_1} S_{i_2} \cdots S_{i_k} S_{i_k}^* S_{i_k} S_{i_k}^* S_{i_{k-1}}^* \cdots S_{i_1}^* = S_{i_1} S_{i_2} \cdots S_{i_k} S_{i_k}^* S_{i_{k-1}}^* \cdots S_{i_1}^*.
\]
From this, it follows that $\phi(\DD(C^*(E_A))) \subseteq \DD(\OO_A)$. 
Since $S_i S_j \neq 0$ if and only if $[ij]$ is an edge in $E_A$, the computation above  also shows that $\phi$ maps $\DD(C^*(E_A))$ onto $\DD(\OO_A)$.
\end{proof}

\begin{theorem}[{cf.\ \cite[Theorem 3.6]{MatsumotoMatui}}] \label{thm: matsumoto matui for graphs}
Let $E,F$ be finite, essential, strongly connected graphs that satisfy Condition (L).
If there exists an isomorphism $\phi \colon C^*(E) \to C^*(F)$ such that $\phi(\DD(C^*(E))) = \DD(C^*(F))$, then 
\[
  \sgn(\det (I - A_E)) = \sgn( \det(I - A_F)).
\]
\end{theorem}

\begin{proof}
In order to apply the result from \cite{MatsumotoMatui}, out-splitting will be used to replace $E$ and $F$ by graphs without multiple edges (see \cite[Section 3]{BatesPask}). 
For each $v \in E^0$, partition $s^{-1}(v)$ into singleton sets. Note that these partitions are proper, and let $E_1$ be the graph obtained from the corresponding out-splitting.
Note that since $E$ is essential, strongly connected and satisfies Condition (L), $E_1$ will be essential, strongly connected and will satisfy Condtion (L).   
Analogously, let $F_1$ be the graph obtained from $F$ by such a complete out-splitting.
By \cite[Corollary 6.2]{BrownloweCarlsenWhittaker}, there exist diagonal preserving isomorphisms $C^*(E) \to C^*(E_1)$ and $C^*(F) \to C^*(F_1)$.
Hence, there also exists a diagonal preserving isomorphism from $C^*(E_1)$ to $C^*(F_1)$.
 
Since $E_1$ and $F_1$ have no multiple edges, satisfy Condition (L), are essential, and strongly connected, $A_{E_1}$ and $A_{F_1}$ are  irreducible $\{0,1\}$ matrices that satisfy Condition (I). 
By Lemma \ref{lem:CKisGraph}, $C^*(E_1) \cong \OO_{A_{E_1}}$ and $C^*(F_1) \cong \OO_{A_{F_1}}$ in a diagonal preserving way, so the argument above guarantees the existence of a diagonal preserving isomophism from $C^*(\OO_{A_{E_1}})$ to $C^*(\OO_{A_{F_1}})$, and hence, it follows from \cite[Theorem 3.6]{MatsumotoMatui} that $\sgn(\det (I - A_{E_1})) = \sgn( \det(I - A_{F_1}))$.
For edge shifts, out-splitting produces a shift space conjugate to the original (see e.g.\ \cite[Theorem 2.4.10]{LindMarcus}). Since flow equivalence is a coarser equivalence relation than conjugacy, this implies that the edge shift of $E$ is flow equivalent to the edge shift of $E_1$. Clearly, the same relation exists between $F$ and $F_1$. By \cite{ParrySullivan}, the sign of the determinant is an invariant of flow equivalence, so  
\[
 \sgn(\det (I - A_E)) = \sgn(\det (I - A_{E_1})) = \sgn( \det(I - A_{F_1})) = \sgn( \det(I - A_F)). \qedhere
\]
\end{proof}

\subsection{Leavitt path algebras over certain subrings of $\C$}\label{subsection:leavitt}

In this section, we will prove an algebraic analogue of Theorem \ref{thm: matsumoto matui for graphs}, i.e.\ a weak algebraic version of Matsumoto and Matui's result. 
We will look only at rings $R$ that are subrings of $\C$ closed under complex conjugation, and we will equip $L_R(E)$ with a \emph{conjugate} linear involution.  
In the following, $\Z$ will be the main example of such a ring.
First, we formalize the connection between $L_{R}(E)$ and $C^*(E)$.

\begin{lemma} \label{lem: LZE into CE}
Let $E$ be a graph and let $R$ be a subring of $\C$ closed under complex conjugation. 
The map $\iota_E \colon L_{R}(E) \to C^*(E)$
given on generators by

\begin{align*}
  \iota_E(v) &= p_v, \quad \text{ for } v \in E^0, \\
  \iota_E(e) &= s_e, \quad \text{ for } e \in E^1.
\end{align*}
extends to a $*$-algebra embedding. 
\end{lemma}
\begin{proof}
By the universal property of $L_{R}(E)$,
$\iota_E$ extends to a $*$-homomorphism. 
It is injective by the Graded Uniqueness Theorem (\cite[Theorem 5.3]{TomfordeLeavittOverRing}), see \cite[Theorem 7.3]{TomfordeUniquenessTheorems} for details on how to apply the Graded Uniqueness Theorem. 
\end{proof}

We now provide a version of \cite[Theorem 4.4]{AbramsTomforde} for Leavitt path algebras over subrings of $\C$
(the result in \cite{AbramsTomforde} is for Leavitt path algebras over $\C$).
This version of the result additionally tracks the image of the diagonal.

\begin{lemma} \label{lem: extending diagonal iso}
Let $R$ be a subring of $\C$ closed under complex conjugation and
let $E,F$ be finite graphs that satisfy Condition (L). 
If $\phi \colon L_R(E) \to L_R(F)$ 
is a $*$-isomorphism such that $\phi(\DD(L_R(E))) \subseteq \DD(L_R(F))$,
then $\phi$ extends to an isomorphism $\overline{\phi} \colon C^*(E) \to C^*(F)$ such that $\overline{\phi}(\DD(C^*(E))) = \DD(C^*(F))$.  
\end{lemma}
\begin{proof}
As in the proof of \cite[Theorem 4.4]{AbramsTomforde}, $\phi$ extends to an isomorphism $\overline{\phi} \colon C^*(E) \to C^*(F)$ satisfying $\overline{\phi} \circ \iota_{E} = \iota_F \circ \phi$.
Hence,
\begin{align*}
  \overline{\phi}(\DD(C^*(E)))	&= \overline{\newspan}_\C\left\{\overline{\phi}(s_\alpha s_\alpha^*) \mid \alpha \in E^*\right\} \\
				&= \overline{\newspan}_\C\left\{\overline{\phi}(\iota_E(\alpha \alpha^*)) \mid \alpha \in E^*\right\} \\
				&= \overline{\newspan}_\C\left\{\iota_F(\phi(\alpha \alpha^*)) \mid \alpha \in E^*\right\} \\
				&= \overline{\newspan}_\C\left\{\iota_F(\phi(\newspan_R\{\alpha \alpha^* \mid \alpha \in E^* \} ))\right\} \\
				&= \overline{\newspan}_\C\left\{\iota_F(\phi(\DD(L_R(E))))\right\} \\
				&\subseteq \overline{\newspan}_\C\left\{\iota_F(\DD(L_R(F)))\right\} \\
				&=  \overline{\newspan}_\C\{ \newspan_R \{ s_\alpha s_\alpha^* \mid \alpha \in F^* \}\} \\
				&= \overline{\newspan}_\C\{ s_\alpha s_\alpha^* \mid \alpha \in F^* \} \\
				&= \DD(C^*(F)).	
\end{align*}
Since $\overline{\phi}$ is an isomorphism and $\DD(C^*(E))$ is a MASA in $C^*(E)$, 
$\overline{\phi}(\DD(C^*(E)))$ is a MASA in $C^*(F)$.
As $\DD(C^*(F))$ is also a MASA in $C^*(F)$, it follows that $\overline{\phi}(\DD(C^*(E))) = \DD(C^*(F))$.
\end{proof}

\begin{theorem} \label{thm: algebraic matsumoto matui}
Let $E,F$ be finite, essential, strongly connected graphs that satisfy Condition (L), and let $R$ be a subring of $\C$ closed under complex conjugation.
If $\phi \colon L_R(E) \to L_R(F)$
is a $*$-isomorphism such that $\phi(\DD(L_R(E))) \subseteq \DD(L_R(F))$,
then 
\[
  \sgn(\det (I - A_E)) = \sgn( \det(I - A_F)).
\]
\end{theorem}
\begin{proof}
By Lemma \ref{lem: extending diagonal iso}, $\phi$ extends to a diagonal preserving isomorphism $\overline{\phi} \colon C^*(E) \to C^*(F)$, so the result follows from Theorem \ref{thm: matsumoto matui for graphs}. 
\end{proof}

\section{Unitaries in $L_{R}(E)$}
In this section, we show that all unitaries in $L_{R}(E)$ can be written in a certain standard from when $E$ is a finite graph and $R$ is \nice\ (see Proposition \ref{prop: standard form unitaries}).
This is similar to what is done in \cite{BS} for unitaries in $L_{\Z}(E_2)$, and the authors are grateful to Chris Smith for showing them how to generalize the arguments given there.
A particularly interesting feature of Smith's argument is that it eschews the presentation of $L_{\Z}(E)$ as endomorphisms on an uncountably generated free module used in \cite{BS} in favor of techniques based on the grading of Leavitt path algebras.  
We generalize Smith's argument to cover all rings that have an \eup.

We first note a simple way to find unitaries in $L_{R}(E)$.

\begin{lemma} \label{lem: orthogonal paths}
Let $E$ be a finite graph and let $R$ be a unital commutative ring with characteristic $0$. 
If $\alpha_1, \alpha_2, \ldots, \alpha_n \in E^*$ satisfy 
\[
 \sum_{i=1}^n \alpha_i \alpha_i^* = 1,
\]
then $\alpha_i^* \alpha_j = 0$ for $i \neq j$. 
\end{lemma}
\begin{proof}
By relabeling, we can assume that $|\alpha_1| \geq |\alpha_2| \geq \cdots \geq |\alpha_n|$.
We see that
\[
	\alpha_1 = 1 \alpha_1 = \left( \sum_{i=1}^n \alpha_i \alpha_i^* \right) \alpha_1 = \sum_{i=1}^n \alpha_i \alpha_i^* \alpha_1.
\]
The first term in the sum on the right hand side is $\alpha_1$, and since $\alpha_1$ has maximal length among the $\alpha_i$, each subsequent term is either a real path or equals $0$. 
Since the real paths are linearly independent (\cite[Proposition 4.9]{TomfordeLeavittOverRing}) and $R$ has characteristic $0$, it follows that $ \alpha_i \alpha_i^* \alpha_1 = 0$, for each $i \neq 1$. 
Hence, $\alpha_i^* \alpha_1 = 0$. 
By taking adjoints, $\alpha_1^* \alpha_i = 0$ for $i = 2, 3, \ldots, n$.

Let $k = 2, 3, \ldots, n$ be given, and assume that $\alpha_i^* \alpha_k = 0$ for all $i < k$.
Then
\[
	\alpha_k = 1 \alpha_k = \sum_{i=1}^n \alpha_i \alpha_i^* \alpha_k = \sum_{i=k}^n \alpha_i \alpha_i^* \alpha_k.
\]   
Since $\alpha_k$ has maximal length among the $\alpha_i$ appearing in the sum on the right hand side, we see as above that $\alpha_i^* \alpha_k = 0 = \alpha_k^* \alpha_i$ for $i = k+1, k+2, \ldots, n$. 
By induction, $\alpha_i^* \alpha_j = 0$ all $i \neq j$. 
\end{proof}

This gives us a simple way to find unitaries in $L_{R}(E)$:

\begin{lemma} \label{lem: certain unitaries}
Let $E$ be a finite graph and let $R$ be a unital commutative ring with characteristic $0$. 
Suppose that $\alpha_i, \beta_i \in E^*$ and $\lambda_i \in R$, $i=1,2,\ldots,n$, are such that
\begin{enumerate}
	\item $\sum_{i=1}^n \alpha_i \alpha_i^* = 1$, \label{item:certain unitaries 1} 
	\item $\sum_{i=1}^n \beta_i \beta_i^* = 1$, \label{item:certain unitaries 2}
\end{enumerate}
and (if $L_{R}(E)$ is equipped with a linear involution)
\begin{enumerate}
\setcounter{enumi}{2}	
\item $\lambda_i^2 = 1$ for all $i$  \label{item:certain unitaries 3}
\end{enumerate}
or (if $L_{R}(E)$ is equipped with a conjugate linear involution)
\begin{enumerate}
\setcounter{enumi}{2}	
\item $\lambda_i \overline{\lambda_i} = 1$ for all $i$. \label{item:certain unitaries 3'}
\end{enumerate}
Then 
\[
	u = \sum_{i=1}^n \lambda_i \alpha_i \beta_i^*,
\]	
is a unitary.
\end{lemma}                       
\begin{proof}
This follows from a simple computation of $uu^*$ and $u^*u$ using Lemma \ref{lem: orthogonal paths}. 
\end{proof}

In general, not all unitaries in $L_{R}(E)$ will have the form described in Lemma \ref{lem: certain unitaries}.
However, in the special case of $R = \Z$ all unitaries do in fact have this form. 
Before proving this result, we note the following simple consequence of the defining relations of a Leavitt path algebra.

\begin{remark} \label{rmk: extend paths}
Let $E$ be a finite graph.
Fix a vertex $v \in E^0$, some natural number $m$, and define 
\[
  X_{v,m} = \{ \gamma \in E^* \mid s(\gamma) = v \text{ and either } |\gamma| = m \text{ or } |\gamma| < m \text{ and } r(\gamma) \text{ is a sink} \}. 
\]
A straightforward induction argument using (\ref{item:summation_relation}) from Definition \ref{def:lpa} shows that
\[
 \sum_{\gamma \in X_{v,m}} \gamma \gamma^* = v. 
\]
Consider  
\[
 x = \sum_{i=1}^n \lambda_i \alpha_i \beta_i^*  \in L_{R}(E), 
\]
with $\lambda_i \in R$ and $r(\alpha_i) = r(\beta_i)$, and let $k = \max\{ |\beta_i| \}$.
Then
\begin{align*}
  x &= \sum_{i=1}^n \lambda_i \alpha_i \beta_i^* = \sum_{i=1}^n \lambda_i \alpha_i r(\alpha_i) \beta_i^* \\
    &= \sum_{i=1}^n \lambda_i \alpha_i \left( \sum_{\gamma \in X_{r(\alpha_i),k-|\beta_i|}} \gamma \gamma^* \right) \beta_i^* \\
    &= \sum_{i=1}^n \sum_{\gamma \in X_{r(\alpha_i),k-|\beta_i|}} \lambda_i \alpha_i \gamma (\beta_i \gamma)^*.
\end{align*}
Re-indexing the last sum, we see that 
\[
  x = \sum_{j=1}^l \hat{\lambda}_j \mu_j \nu_j^*,
\]
where each $\mu_j$ extends an $\alpha_i$, each $\nu_j$ extends a $\beta_i$, and all the $\nu_j$ have length $k$ or are shorter but end at a sink.
In particular, if $\nu_j \neq \nu_i$ then $\nu_j^* \nu_i = 0$.
\end{remark}

We are now ready to show that all unitaries in $L_{R}(E)$ can be written in the form from Lemma \ref{lem: certain unitaries} when $R$ is \nice.  
The following proposition is due to Chris Smith in the case where $R = \Z$. 

\begin{proposition} \label{prop: standard form unitaries}
Let $E$ be a finite graph, let $R$ be \nice, and let $u \in L_{R}(E)$ be a unitary.
Then there exist paths $\alpha_i, \beta_i$ such that
\[
	u = \sum_{i=1}^n \lambda_i \alpha_i \beta_i^*,
\]
with 
\begin{enumerate}
 \item $\sum_{i=1}^n \alpha_i \alpha_i^* = 1$, \label{item:standard form alpha}
 \item $\sum_{i=1}^n \beta_i \beta_i^* = 1$, and, \label{item:standard form beta} 
 \item $|\lambda_i| = 1$. \label{item:standard form lambda}
\end{enumerate}
\end{proposition}
\begin{proof}
By Remark \ref{rmk: extend paths},
\[
	u = \sum_{i=1}^n \lambda_i \alpha_i \beta_i^*
\]
for some $\lambda_i \in R$
and paths $\alpha_i, \beta_i$ with $r(\alpha_i) = r(\beta_i)$ such that the $\beta_i$ are all of some fixed length $t$ or shorter but ending in a sink. 
By combining like terms and dropping terms where $\lambda_i = 0$, we may assume that the pairs $(\alpha_i, \beta_i)$ are distinct and that each $\lambda_i \neq 0$.

We will now show that the $\beta_i$ are distinct and incidentally that $\lambda_i = \pm 1$ for each $i$.
Fix an index $1 \leq l \leq n$.
We have 
\begin{align*}
  r(\beta_l) &= \beta_l^* \beta_l = \beta_l^* u^*u \beta_l = \beta_l^* \left( \left(\sum_{i=1}^n \lambda_i \beta_i \alpha_i^* \right) \left( \sum_{j=1}^n \overline{\lambda_j} \alpha_j \beta_j^* \right) \right) \beta_l \\
    &= \beta_l^* \left( \sum_{i,j= 1}^n \lambda_i \overline{\lambda_j} \beta_i \alpha_i^* \alpha_j \beta_j^* \right) \beta_l
    = \sum_{i,j= 1}^n \lambda_i \overline{\lambda_j} \beta_l^* \beta_i \alpha_i^* \alpha_j \beta_j^* \beta_l.
\end{align*}
By construction of the $\beta_i$, we see that $\beta_i^* \beta_j = 0$ if $\beta_i \neq \beta_j$. 
To ease notation, let $L = \{i \mid \beta_i = \beta_l\}$.
Continuing the computation we get 
\begin{align} \label{eqn: r(beta_l)}
 r(\beta_l) &= \sum_{i,j \in L} \lambda_i \overline{\lambda_j} \beta_l^* \beta_l \alpha_i^* \alpha_j \beta_l^* \beta_l = \sum_{i,j \in L} \lambda_i \overline{\lambda_j} \alpha_i^* \alpha_j,
\end{align}
where the last equality follows from $\beta_l^* \beta_l = r(\beta_l) = r(\alpha_i)$ for $i \in L$.

Recall that all Leavitt path algebras are $\Z$-graded, that vertex projections are homogeneous of degree $0$, and elements of the form $\mu \nu^*$ are homogeneous of degree $|\mu| - |\nu|$ (\cite[Definition 4.5 and Proposition 4.7]{TomfordeLeavittOverRing}).
Since the left hand side of (\ref{eqn: r(beta_l)}) is homogeneous of degree $0$ the right hand side must be so too. 
Hence, we can discard any element not of degree $0$ from the sum. 
Since the pairs $(\alpha_i, \beta_i)$ are distinct, and since we only sum over indices $i,j$ with $\beta_i = \beta_l = \beta_j$, we see that the $\alpha_i$ appearing in the sum must be distinct. 
Thus, for $i \neq j$ the term $\alpha_i^* \alpha_j$ is either $0$ or homogeneous of degree $|\alpha_i| - |\alpha_j|$ (note that if $|\alpha_i| = |\alpha_j|$ then $\alpha_i^* \alpha_j = 0$ since the paths are distinct). 
Thus, we get  
\[
 r(\beta_l) = \sum_{i \in L} |\lambda_i|^2  \alpha_i^* \alpha_i.
\]
Using that $\alpha_i^* \alpha_i = r(\alpha_i) = r(\beta_l)$, it follows that
\[
 r(\beta_l) = \left( \sum_{i \in L} |\lambda_i|^2 \right) r(\beta_l).
\]
Since each $\lambda_i$ is non-zero and $R$ has an \eup, this equality can only hold if $|L| = 1$ and $|\lambda_l| = 1$.
Hence, all the $\beta_i$ are distinct and $|\lambda_i| = 1$ for each $i$, i.e.\ (\ref{item:standard form lambda}) holds.

Since the $\beta_i$ are distinct, $\beta_i^* \beta_j = 0$ if $i\neq j$, and as always $\beta_i^* \beta_i = r(\beta_i)$.
Thus,
\begin{align*}
  1 &= uu^* = \left( \sum_{i=1}^n \lambda_i \alpha_i \beta_i^* \right) \left( \sum_{j=1}^n \overline{\lambda_i} \beta_j \alpha_j^* \right) \\
    &= \sum_{i=1}^n |\lambda_i|^2 \alpha_i \alpha_i^* = \sum_{i=1}^n \alpha_i \alpha_i^*.
\end{align*}
That is, (\ref{item:standard form alpha}) holds. 
By Lemma \ref{lem: orthogonal paths}, (\ref{item:standard form alpha}) implies $\alpha_i^* \alpha_j = 0$ if $i \neq j$, so  
\[
  1 = u^* u = \left( \sum_{j=1}^n \lambda_i \beta_j \alpha_j^* \right) \left( \sum_{i=1}^n \overline{\lambda_i} \alpha_i \beta_i^* \right) = \sum_{i=1}^n \beta_i \beta_i^*.
\]
Hence, (\ref{item:standard form beta}) also holds.
\end{proof}

In the case of a finite graph with no sinks, we can give a completely geometric description of when a set of paths $\alpha_1, \alpha_2, \ldots, \alpha_n$ satisfy that the projections $\alpha_i \alpha_i^*$ sum to $1$. 

\begin{lemma} \label{lem: disjoint cylinder sets}
Let $E$ be a finite graph with no sinks, let $R$ be a unital commutative ring with characteristic $0$, and let $\alpha_i$, $i=1,2,\ldots,n$, be paths in $E$. 
The following are equivalent
\begin{enumerate}
 \item $\bigsqcup_{i=1}^n \ZZ(\alpha_i) = E^{\infty}$.\label{item:cylinders}
 \item $\sum_{i=1}^n \alpha_i \alpha_i^* = 1$.\label{item:sum}
\end{enumerate}
\end{lemma}
\begin{proof}

By repeated applications of (\ref{item:summation_relation}) from Definition \ref{def:lpa}, it follows that $(\ref{item:cylinders})$ implies $(\ref{item:sum})$.
Suppose now that $(\ref{item:sum})$ holds.
We first show that $\cup_i \ZZ(\alpha_i) = E^{\infty}$. 
Suppose for contradiction that there is some $\mu \in E^{\infty}$ that is not in $\sqcup_i \ZZ(\alpha_i)$, and let $\nu$ be the initial segment of $\mu$ of length $\max{|\alpha_i|}$. 
Then $\alpha_i^* \nu = 0$ for all $i$ and therefore 
\[
 \nu = 1\nu = \left(\sum_{i=1}^n \alpha_i \alpha_i^* \right) \nu = \sum_{i=1}^n \alpha_i \alpha_i^* \nu = 0,
\]
a contradiction.
By Lemma \ref{lem: orthogonal paths}, we have $\alpha_i^* \alpha_j = 0$ for $i \neq j$.
This implies that no $\alpha_i$ is an initial segment of any $\alpha_j$, and hence that the cylinder sets are disjoint. 
\end{proof}

\section{Projections in $L_{R}(E)$}

The aim of this section is to study the projections in $L_{R}(E)$
for a finite graph $E$ and $R$ \nice. 
Specifically, it will be proved that all projections in $L_{R}(E)$
are elements of the diagonal.

\begin{definition}\label{def: diagonal projections}
Let $E$ be a graph. 
A projection $p \in L_R(E)$ is said to be \emph{diagonal} if $p \in \DD(L_{R}(E))$.
\end{definition}

\noindent The following Example gives a basic illustration of a Leavitt path algebra where all projections are diagonal.

\begin{example}
Consider the graph 

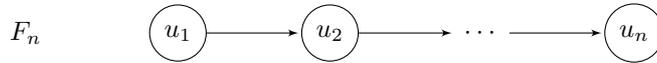
\begin{center}
\begin{tikzpicture}[shorten >=1pt]

\tikzset{vertex/.style = {shape=circle,draw,minimum size=2em}}
\tikzset{edge/.style = {->,> = latex'}}

\node at (0,0) {$F_n$};

\node[vertex] (u1) at (2,0) {$u_1$};
\node[vertex] (u2) at (4,0) {$u_2$};
\node (u3) at (6,0) {$\cdots$};
\node[vertex] (un) at (8,0) {$u_n$};

\draw[edge] (u1) to (u2);
\draw[edge] (u2) to (u3);
\draw[edge] (u3) to (un);

\end{tikzpicture}
\end{center}

\noindent It is well known that $L_{\Z}(F_n) \cong M_n(\Z)$, and furthermore that the diagonal is exactly the diagonal matrices. 
We claim that $L_{\Z}(F_n)$ only has diagonal projections. 

Suppose 
\[
  P = \begin{pmatrix}
       a_{11} & a_{12} & \cdots & a_{1n} \\
       a_{21} & a_{22} & \cdots & a_{2n} \\
       \vdots & \vdots & \ddots & \vdots \\
       a_{n1} & a_{n2} & \cdots & a_{nn}
      \end{pmatrix}
\]
is a projection in $M_n(\Z)$.
The relations $P = P^*$ and $P = P^2$ imply that $P = PP^T$, so for every $i$ we have
\[
  a_{ii} = \sum_{j=1}^n a_{ij}^2.
\]
This can only be satisfied if $P$ is a diagonal matrix with entries in $\{0,1\}$.
\end{example}

The following two lemmas will pave the way for the proof that all projections in $L_{\Z}(E)$ are diagonal.

\begin{lemma} \label{lem: doubling}
Let $E$ be a graph and let $R$ be \nice. 
If $x,y \in L_{R}(E)$
satisfy $2x = 2y$ then $x = y$. 
\end{lemma}
\begin{proof}
Let $\iota_E \colon L_{R}(E) \to C^*(E)$
be the embedding from Lemma \ref{lem: LZE into CE}. 
We have
\[
  2\iota_E(x) = \iota_E(2x) = \iota_E(2y) = 2\iota_E(y),
\]
so since $2$ is invertible in $\C$, we get $\iota_E(x) = \iota_E(y)$. 
Because $\iota_E$ is an embedding, $x = y$.
\end{proof}

\begin{remark} \label{rmk: add tail}
Let $E$ be a graph, $R$ a commutative unital ring. 
Using a process called ``adding tails'' (see \cite{BatesPaskRaeburnSzymanski}),  we can find a graph $F$ with no sinks such that $L_{R}(E)$ embeds into $L_{R}(F)$ as $*$-algebras, furthermore this embedding maps $E^*$ into $F^*$. 

For Leavitt path algebras over fields this is proved in \cite[Section 5]{AbramsArandaPino08}, where infinite emitters are also ``desingularized'' . 
We are only interested in removing sinks, so we preform a partial desingularization, we are also interested in working over general rings, not just fields, however the arguments for constructing the embedding given in \cite{AbramsArandaPino08} works perfectly well in this setting, so we only review them briefly. 
We define $F$ as follows
\begin{align*}
  F^0 &= E^0 \sqcup \{ w_{i} \mid i \in \N, w \in E^0, w \text{ is a sink} \}, \\
  F^1 &= E^1 \sqcup \{ e^{w}_i \mid i \in \N, w \in E^0, w \text{ is a sink} \},
\end{align*}
we let the range and source maps extend those of $E$ and define $r(e^{w}_i) = e^{w}_{i+1}$ and 
\[
  s(e^{w}_i) = \begin{cases}
                w_{i-1}, & \text{if } i > 1 \\
                w, &  \text{otherwise}
               \end{cases}.
\]
Now we can define a $*$-homomorphism $\phi \colon L_{R}(E) \to L_{R}(F)$ by stipulating that $\phi(v) = v$ and $\phi(e) = e$ for all $v \in E^0$ and $e \in E^1$. 
By the Graded Uniqueness Theorem (\cite[Theorem 5.3]{TomfordeLeavittOverRing}) $\phi$ is injective. 
\end{remark}

\begin{lemma} \label{lem: sum of paths}
Let $E$ be a graph, let $R$ be \nice, let $x \in L_{R}(E)$ and let $\lambda \in R$. 
If $\alpha, \beta \in E^*$ are such that 
\[
  2 x = \alpha + \lambda \beta, 
\]
then $\alpha = \beta$.
\end{lemma}
\begin{proof}
By Remark \ref{rmk: add tail}, we can assume that $E$ has no sinks. 
We may write 
\[
 x = \sum_{i=1}^n \lambda_i \mu_i \nu_i^*,
\]
where $\lambda_i \in R$
$\mu_i, \nu_i \in E^*$. 
Pick a path $\gamma$ based at $r(\alpha)$ such that $|\gamma|$ is greater than $\max\{ |\nu_i| \}$. 
Then $x \gamma$ is a polynomial in real edges,
\[
  x \gamma = \sum_{k=1}^m \hat{\lambda}_k \xi_k,
\]
where $\hat{\lambda}_k \in R$
and $\xi_k \in E^*$. 
Thus,  
\[
  \sum_{k=1}^m 2 \hat{\lambda}_k \xi_k = 2x\gamma = \alpha \gamma + \lambda \beta \gamma.
\]
The real paths form a linearly independent set in $L_{R}(E)$ (\cite[Proposition 4.9]{TomfordeLeavittOverRing}), so if $\alpha \gamma \neq \beta \gamma$, it follows from the above equation that there is some subset $I \subseteq \{1,2,\ldots,m\}$ such that 
\[
 1 = \sum_{k \in I} 2 \hat{\lambda}_k = 2 \left( \sum_{k \in I} \hat{\lambda}_k \right),
\]
which contradicts that $\frac{1}{2}\notin R$ (Remark \ref{rmk:PropX_nohalf}). 
Therefore, we must have $\alpha \gamma = \beta \gamma$.
Since $s(\gamma) = r(\alpha)$, we have $\alpha \gamma \neq 0$ and hence we must have $\beta \gamma \neq 0$ so $s(\gamma) = r(\beta)$.
Thus $\alpha = \beta$.
\end{proof}

\begin{theorem} \label{thm: projections look nice}
Let $E$ be a finite graph and let $R$ be \nice.
If $p \in L_{R}(E)$
is a projection then
\[
  p = \sum_{i=1}^n \beta_i \beta_i^*,
\]
for some paths $\beta_i \in E^*$ with $\beta_i^* \beta_j = 0$ for $i \neq j$. 
\end{theorem}
\begin{proof}
Put $u = 2p - 1$. 
Then $u$ is a self-adjoint unitary, i.e.\ $u = u^*$ and $u^2 = 1$. 
By Proposition \ref{prop: standard form unitaries}, we can write 
\[
  u = \sum_{i=1}^n \lambda_i \alpha_i \beta_i^*,
\]
where $|\lambda_i| = 1$,
$\alpha_i, \beta_i \in E^*$ and $\sum_{i=1}^n \alpha_i \alpha_i^* = 1 = \sum_{i=1}^n \beta_i \beta_i^*$.
Thus,
\[
  2p = 1 + u = \sum_{i=1}^n \beta_i \beta_i^* + \sum_{i=1}^n \lambda_i \alpha_i \beta_i^* 
     = \sum_{i=1}^n \left( \beta_i + \lambda_i \alpha_i \right) \beta_i^*.
\]	

Let $k \in \{1,2,\ldots, n\}$ be given.
By Lemma \ref{lem: orthogonal paths} $\beta_i^* \beta_k = 0 $ when $i \neq k$, so 
\[
  2p\beta_k = \beta_k + \lambda_k \alpha_k.
\]
By Lemma \ref{lem: sum of paths}, this implies that $\alpha_k = \beta_k$. 
So 
\[
  u = \sum_{i=1}^n \lambda_i \beta_i \beta_i^*.
\]
We then get
\[
  \beta_k = \left( u^2 \right) \beta_k = \left( \sum_{i=1}^n \lambda_i \beta_i \beta_i^* \right) \left( \sum_{j=1}^n \lambda_j \beta_j \beta_j^* \right) \beta_k = \left( \sum_{i=1}^n \lambda_i^2 \beta_i \beta_i^* \right) \beta_k = \lambda_k^2 \beta_k.
\]
Since the real paths are linearly independent (\cite[Proposition 4.9]{TomfordeLeavittOverRing}), we get that $\lambda_k = \pm 1$. 
It follows that
\[
  2 p = \sum_{i=1}^n \varepsilon_i \beta_i \beta_i^*,
\]
where $\varepsilon_i \in \{0,2\}$.
An application of Lemma \ref{lem: doubling} completes the proof. 
\end{proof}

\noindent The following result is an immediate consequence of the preceding theorem. 

\begin{corollary} \label{cor: only diagonal projections}
Let $R$ be \nice.
If $E$ is a finite graph, then $L_{R}(E)$ only has diagonal projections. 
\end{corollary}

\noindent 
It is not clear that this result only holds when $R$ is \nice, however it is worth noting that the result certainly requires some restriction on the ring of coefficients.
For instance,
\[
  p = \frac{1}{2}\left( aa^* + ab^* + ba^* + bb^*\right)
\]
is non-diagonal projection in $L_{2,\C}$.

\section{Conclusions}

\begin{proposition}\label{prop: only dynamical homomorphisms}
Let $E,F$ be graphs, let $R$ be a subring of $\C$ closed under complex conjugation, and let $\phi \colon L_{R}(E) \to L_{R}(F)$ be a $*$-homomorphism. 
If $L_{R}(F)$ only has diagonal projections, then $\phi(\DD(L_{R}(E))) \subseteq \DD(L_{R}(F))$. 
\end{proposition}
\begin{proof}
Let $\alpha \in E^*$ be given. 
Since $\alpha \alpha^*$ is a projection in $L_{R}(E)$ and since $\phi$ is a $*$-homomorphism, $\phi(\alpha \alpha^*)$ is a projection. 
By the assumption on $L_{R}(F)$, it follows that $\phi(\alpha \alpha^*) \in \DD(L_{R}(F))$.
By linearity of $\phi$ and the fact that $\DD(L_{R}(F))$ is an algebra, we get that $\phi(\DD(L_{R}(E))) \subseteq \DD(L_{R}(F))$.
\end{proof}

\begin{remark}
The requirement that $L_{R}(F)$ only has diagonal projections is clearly crucial for the result.
For instance, there exist $*$-automorphisms of $L_{2,\C}$ that do not preserve the diagonal.
As an example of this, consider the unitary 
\[
  u = \frac{1}{\sqrt{2}} \left( aa^* - ab^* + ba^* + bb^* \right),
\]
and define an automorphism $\psi$ of $L_{2,\C}$ by $\psi(x) = u x u^*$.
Then 
\[
  \psi(aa^*) = \frac{1}{2}\left( aa^* + ab^* + ba^* + bb^*\right),
\]
so $\psi$ does not preserve the diagonal. 
This $u$ was constructed by mapping the unitary that rotates by $45$ degrees from $M_2(\C)$ into $L_{2,\C}$ using the map discussed in \cite[Example 5.13]{BS}.
\end{remark}

\begin{theorem} \label{thm: isomorphism same sign}
Let $E,F$ be finite, essential, strongly connected graphs that satisfy Condition (L), and let $R$ be a subring of $\C$ closed under complex conjugation.
Assume furthermore that $L_{R}(F)$ only has diagonal projections. 
If $L_{R}(E)$ is $*$-isomorphic to $L_{R}(F)$ then 
\[
  \sgn(\det (I - A_E)) = \sgn( \det(I - A_F)).
\]
\end{theorem}
\begin{proof}
Suppose $\phi \colon L_{R}(E) \to L_{R}(F)$
is a $*$-isomorphism.  
By Proposition \ref{prop: only dynamical homomorphisms}, we must have $\phi(\DD(L_{R}(E))) \subseteq \DD(L_{R}(F))$, so by Theorem \ref{thm: algebraic matsumoto matui}, 
\[
  \sgn(\det (I - A_E)) = \sgn( \det(I - A_F)). \qedhere
\]
\end{proof}

\begin{corollary} \label{cor: iso implies same sign}
Let $E,F$ be finite, essential, strongly connected graphs that satisfy Condition (L) \and let $R$ be \nice.
If $L_{R}(E)$ is $*$-isomorphic to $L_{R}(F)$ then
\[
  \sgn(\det (I - A_E)) = \sgn( \det(I - A_F)).
\]
\end{corollary}
\begin{proof}
This an immediate consequence of Theorem \ref{thm: isomorphism same sign}, since all projections in $L_{R}(E)$ are diagonal by Corollary \ref{cor: only diagonal projections}. 
\end{proof}

\begin{corollary} \label{cor: main result}
Let $R$ be \nice. 
Then $L_{2,R}$ is not $*$-isomorphic to $L_{2-,R}$.
In particular, $L_{2,\Z}$ is not $*$-isomorphic to $L_{2-,\Z}$. 
\end{corollary}
\begin{proof}
Recall that by definition $L_{2,R} = L_{R}(E_2)$ and $L_{2-,R} = L_{R}(E_{2-})$.
We see that the adjacency matrices are 
\[
  A_{E_2} = \begin{pmatrix}
            2
           \end{pmatrix},
\]
and 
\[
  A_{E_{2-}} = \begin{pmatrix}
                2 & 1 & 0 \\
                1 & 1 & 1 \\
                0 & 1 & 1
               \end{pmatrix}.
\]
Hence, $\det(I - A_{E_2}) = -1$ and $\det(I - A_{E_{2-}}) = 1$, so by Corollary \ref{cor: iso implies same sign}, $L_{2,\Z} \not \cong L_{2-,\Z}$
as $*$-algebras.
\end{proof}

Although Corollary \ref{cor: main result} does not apply to fields, we can use it to say something about the situation for fields.

\begin{remark}\label{remark:fields}
Suppose $K$ is a field with characteristic $0$.
Then $L_{2,\Z}$ will embed into $L_{2,K}$, and $L_{2-,\Z}$ will embed into $L_{2-,K}$. 
Assume that there exists a $*$-isomorphism $\phi \colon  L_{2-,K} \to  L_{2,K}$ which restricts to a $*$-homomorphism from $L_{2-,\Z}$ to $L_{2,\Z}$.
The restriction is also injective, so by Corollary \ref{cor: main result}, it cannot be surjective.

One now wonders how common it is for $*$-isomorphisms between Leavitt path algebras over $K$ to restrict to $*$-homomorphisms between the corresponding algebras over $\Z$. 
In general, homomorphisms will not restrict in this way. For instance, each $z \in \C$ of modulus one defines a \emph{gauge automomorphism} $\gamma_z \colon L_{2,\C} \to L_{2,\C}$ given by $\gamma_z(a) = za$ and $\gamma_z(b) = zb$ which clearly does not restrict to an automorphism of $L_{2,\Z}$.
However, in many cases where homomorphisms are written down explicitly, the field of coefficients has not been used, and such homomorphisms must clearly restrict to homomorphisms over $\Z$.
For some examples of this phenomenon, see for instance \cite{AbramsAnhPardo, AbramsLoulyPardoSmith, BS, RuizTomforde}.
In fact, at the end of \cite{TomfordeLeavittOverRing}, Tomforde suggests that Leavitt path algebras over $\Z$ may provide the key to understanding this non-dependence on the field of coefficients. 

Our result leaves two ways for $L_{2-,K}$ and $L_{2,K}$ to be $*$-isomorphic: Either the isomorphism makes explicit use of the field, so that it does not induce a $*$-homomorphism from $L_{2-,\Z}$ to $L_{2,\Z}$.
Or, if the isomorphism does induce such a $*$-homomorphism, then this induced map cannot be surjective.
Either way, a $*$-isomorphism between $L_{2-,K}$ and $L_{2,K}$, if it exists, will have to be a fairly complicated map.
\end{remark}

We conclude with a few remarks on the possibility that $L_{2,R}$ and $L_{2-,R}$ are Morita equivalent. 

\begin{remark}
\label{remark:morita}
Let $K$ be a field. 
If $L_{2,K}$ and $L_{2-,K}$ are Morita equivalent, they must be isomorphic as rings (this follows from the argument given in the proof of $(2)$ implies $(1)$ in \cite[Proposition 10.4]{RuizTomforde}).
It is also worth noting that there are no known examples of Leavitt path algebras that are isomorphic as rings, but not as $*$-algebras.
See \cite{AbramsTomforde} for a discussion of how various notions of isomorphism of Leavitt path algebras are related to isomorphism of graph $C^*$-algebras.

When working over a commutative, unital ring $R$, the connection between Morita equivalence and ring isomorphism of $L_{2,R}$ and $L_{2-,R}$ is less clear. 
If $R$ is a regular supercoherent ring, then both $L_{2,R}$ and $L_{2-,R}$ have trivial algebraic $K$-theory \cite{AraBrustengaCortinas}.
So at least this does not provide an obvious proof that they are not Morita equivalent.
We note that the class of regular supercoherent rings covers the class of Noetherian regular rings which in turn contains the class of principal ideal domains.

To go from a Morita equivalence of $L_{2,K}$ and $L_{2-,K}$ to a ring isomorphism, one uses that $L_{2,K}$ is simple.
In particular, the description of the $K_0$-group of a simple purely infinite ring given in \cite[Corollary 2.2]{AraGoodearlPardo} is crucial.
However, if the ring $R$ has non-trivial ideals, then $L_{2,R}$ will clearly not be simple.
Therefore, it is unclear to the authors if Morita equivalence of $L_{2,R}$ and $L_{2-,R}$ implies ring isomorphism.
\end{remark}

\section*{Acknowledgements}
As mentioned above, the proof of Proposition \ref{prop: standard form unitaries} is due to Chris Smith when the ring is $\Z$, and the authors are very grateful to him for sharing this result. 
After the second named author gave talk at the 2015 Norwegian Operator Algebras Meeting about the case $R = \Z$, emphasizing the property of $\Z$ used in the proof (having an \eup), it was pointed out to the authors by Christian Skau and Lars Tuset that many rings have this property.
The authors thank them for this valuable insight.
Additionally, the authors would like to thank Gene Abrams, Mark Tomforde, and Efren Ruiz for their comments on an early draft of this paper. 

The authors' collaboration began at the conference \emph{Classification of $C^*$-algebras, flow equivalence of shift spaces, and graph and Leavitt path algebras} hosted at the University of Louisiana at Lafayette, and the authors are grateful to the organizers and to the National Science Foundation for providing this opportunity.
This work was supported by VILLUM FONDEN through the experimental mathematics network at the University of Copenhagen and by the Danish National Research Foundation through the Centre for Symmetry and Deformation (DNRF92).



\end{document}